\documentclass[11pt, english]{amsart}

\usepackage[utf8x]{inputenc}

\usepackage[T1]{fontenc}

\usepackage{libertine}

\usepackage[style=british]{csquotes}

\usepackage{epigraph}
\usepackage{xcolor}
\colorlet{darkteal}{teal!70!black}

\usepackage[naturalnames=true, pageanchor=true, hyperfootnotes=false, breaklinks, colorlinks, allcolors=darkteal]{hyperref}

\theoremstyle{plain}
\newtheorem{proposition}{Proposition}
\newtheorem{theorem}[proposition]{Theorem}
\newtheorem{corollary}[proposition]{Corollary}
\newtheorem{lemma}[proposition]{Lemma}
\theoremstyle{definition}
\newtheorem{remark}[proposition]{Remark}

\DeclareMathOperator{\Homeo}{Homeo} 
\DeclareMathOperator{\Isom}{Isom} 
\DeclareMathOperator{\Sym}{Sym} 
\DeclareMathOperator{\Aut}{Aut} 
\DeclareMathOperator{\Fixp}{Fix} 
\DeclareMathOperator{\Fixa}{Fix} 
\DeclareMathOperator{\dom}{dom} 

\newcommand{\Contrucb}{\mathcal{C}_{\mathrm{rucb}}}

\newcommand{\Reels}{\mathbf{R}}
\newcommand{\Nat}{\mathbf{N}}
\newcommand{\Int}{\mathbf{Z}}
\renewcommand{\phi}{\varphi}
\newcommand{\clr}{\mathbf{L}_{\geq 0}}

\newcommand{\abs}[1]{\left| #1 \right|}
\newcommand{\CB}{\mathrm{CB}}
\newcommand{\propum}{(\textasteriskcentered)}

\begin{document}

\title[The homeomorphism group of the first uncountable ordinal]{The homeomorphism group\\ of the first uncountable ordinal}
\author{Maxime Gheysens}
\address{Institut für Geometrie, Technische Universität Dresden, 01062 Dresden, Germany}
\curraddr{Institut für Diskrete Mathematik und Algebra, Technische Universität Bergakademie Freiberg, 09596 Freiberg, Germany}
\email{maxime.gheysens@math.tu-freiberg.de}
\thanks{Work supported in part by the European Research Council Consolidator Grant no.~681207.}
\date{3 December 2019 (version 1), 30 November 2020 (current version)}

\begin{abstract}
 We show that the topology of pointwise convergence on scattered spaces is compatible with the group structure of their homeomorphism group. We then establish a few topological properties of the homeomorphism group of the first uncountable ordinal, such as amenability and Roelcke-precompactness.
\end{abstract}

\maketitle

\tableofcontents

It is easily shown that the topology of pointwise convergence on a discrete space $X$ is compatible with the group structure of the full permutation group $\Sym(X)$---when $X$ is countable, this actually gives rise to a much-studied Polish group. On more general topological spaces, there is no such compatibility for the homeomorphism group\footnote{An explicit counterexample for $\Homeo(\Reels^2)$ can be found for instance in \cite[X, \S~3, exerc.~14]{Bourbaki_TG_V}.}, unless we restrict our attention to equicontinuous subgroups (see~\cite[X, \S~3, n°~5]{Bourbaki_TG_V}).

The first part of this note shows that there is such a compatibility for the homeomorphism group of \emph{scattered} spaces (we recall the definition below). The first uncountable ordinal $\omega_1$, endowed with the order topology, is such a space, and we study its homeomorphism group in the second part of this note.

This group features an intriguing combination of topological properties. On the one hand, $\Homeo(\omega_1)$ may seem \enquote{tame}, as it enjoys such rigidity properties as \emph{amenability} and \emph{Roelcke-precompactness} (Theorems~\ref{th:amen} and~\ref{th:roelcke}, respectively). In particular, it fixes a point whenever it acts affinely on a nonempty compact convex set and it has bounded orbits whenever it acts isometrically on a metric space. On the other hand, this group may look unwiedly, as it is not Baire (Section~\ref{sec:baire}) and admits no nontrivial morphism to any metrisable group (Theorem~\ref{th:propum}).

We conclude this paper by two appendices of independent interest about the automorphism group of two other uncountable structures, which serve us as contrasting examples to $\Homeo(\omega_1)$. The first one shows that $\Sym(X)$ is Baire whatever the cardinality of the set $X$. The second one shows that, when $T$ is a regular tree of uncountable degree, $\Aut^+(T)$, the group of automorphisms of $T$ generated by the fixators of edges, also admits no nontrivial morphism to any metrisable group, despite having a more varied class of isometric actions on metric spaces than $\Homeo(\omega_1)$.

Our aim in this paper is to study the homeomorphism group of one of the most interesting scattered spaces, $\omega_1$, and we do so by leveraging its peculiar structure. In a companion paper \cite{Gheysens_dynscatt}, which can be read independently, we prove via another, more topological, approach, that amenability and Roelcke-precompactness hold actually for all zero-dimensional scattered spaces. Of course, this generality does not allow to retrieve the \enquote{uncountable features} of $\Homeo(\omega_1)$; in exchange, we get two results that are not covered here: the classification of all closed normal subgroups and the computation of the universal minimal flow \cite[\S~3]{Gheysens_dynscatt}.

\subsubsection*{Acknowledgements}

The author is thankful to Pierre de la Harpe and Nicolas Monod whose comments on an earlier draft led to numerous clarifications which hopefully improved the readability of this note. The referee's remarks also led to some elucidations and were instrumental in the resolution to write the companion paper.

\section{The pointwise convergence in scattered spaces}

A topological space $X$ is said \emph{scattered} if any nonempty subset $A$ of $X$ contains an isolated point (with respect to $A$)---in other words, $X$ does not contain any nonempty perfect subspace. The easiest nondiscrete examples are given by ordinals endowed with the order topology: any nonempty subset $A$ contains a minimum, which is isolated in $A$ for the order topology.

For a topological space $X$ and an ordinal $\alpha$, we write $X^{(\alpha)}$ for the $\alpha$-th derived subspace of $X$, defined inductively by $X^{(0)} = X$, $X^{(\alpha + 1)} = \{\text{limit points of }X^{(\alpha)}\}$, and $X^{(\lambda)} = \bigcap_{\alpha < \lambda} X^{(\alpha)}$ for $\lambda$ a limit ordinal. This decreasing sequence of subspaces of $X$ stabilises at some ordinal $\CB(X)$, the so-called \emph{Cantor--Bendixson rank} of $X$. Clearly, $X$ is scattered if and only if $X^{(\CB(X))} = \emptyset$. In that case, for any $x \in X$, there is a (unique) ordinal $\alpha$ such that $x \in X^{(\alpha)} \setminus X^{(\alpha + 1)}$.

Since scattered spaces are somehow \enquote{exhaustible by isolated points}, the following result is not surprising.

\begin{proposition}\label{prop:topscattpw}
	Let $X$ be a scattered space. On $\Homeo(X)$, the topology of pointwise convergence \emph{agrees} with the topology of discrete pointwise convergence (that is, pointwise convergence on the set $X$ endowed with the discrete topology).
\end{proposition}

\begin{proof}
	Since the former topology is obviously weaker than the latter one, we only need to show that if a net $(g_j)$ of homeomorphisms of $X$ converges pointwise to some homeomorphism $g$, then it also converges pointwise with respect to the discrete topology. That is, we need to show that, for any $x \in X$, $(g_j (x))$ is eventually equal to $g(x)$. Let $\alpha$ be the (unique) ordinal such that $x \in X^{(\alpha)} \setminus X^{(\alpha + 1)}$. Since homeomorphisms preserve the derived subspaces of $X$, $g_j(x)$ as well as $g (x)$ also belong to $X^{(\alpha)} \setminus X^{(\alpha + 1)}$. But, by construction, the latter is a discrete subspace of $X$. Therefore the convergence of $(g_j (x))$ to $g(x)$ is only possible if $(g_j (x))$ is eventually equal to $g(x)$.
\end{proof}

\begin{corollary}\label{cor:homeoscatt}
	If $X$ is a scattered space, then $\Homeo(X)$, endowed with the topology of pointwise convergence, is a topological group admitting an identity neighbourhood basis made of open subgroups (i.e.~a \emph{nonarchimedean} group).
\end{corollary}

Indeed, we can take the subgroups fixing a finite set as basic neighbourhoods. The group $\Homeo(X)$ is then in particular a totally disconnected group.

Let us mention, however, that this compatibility with the group structure does \emph{not} ensure that the action map $\Homeo(X) \times X \rightarrow X$ is continuous (see Remark~\ref{rem:actionnotcontinuous}).

\begin{remark}\label{rem:pwunifcomp}
	Being scattered is of course not a necessary condition for the topology of pointwise convergence to be compatible with the group structure of $\Homeo(X)$. To say nothing about spaces with very few homeomorphisms (see e.g.~\cite{vanMill_1983}), the topology of pointwise convergence on $\Homeo(\Reels)$ happens to agree with the topology of uniform convergence on compact subsets\footnote{Indeed, let $a < b$ and $c < d$ be points in $\Reels$ and $W$ be the set of homeomorphisms that send the compact interval $[a, b]$ into the open one $]c, d[$. We need to show that $W$ is open for the topology of pointwise convergence. For any $a' \neq b' \in ]c, d[$, let $\delta$ be the minimum of the four possible distances between $a', b'$ and $c, d$. By the intermediate value theorem, the image of $[a, b]$ under a homeomorphism $g$ is an interval whose extremities are $g(a)$ and $g(b)$. Therefore, the set $U({a', b', \delta})$ of all homeomorphisms such that $\abs{g(a) - a'} < \delta, \abs{g(b) - b'} < \delta$, which is open for the topology of pointwise convergence, is included in $W$. Moreover, $W$ is the union of all such $U({a', b', \delta})$, when $a'\neq b'$ range among $]c, d[$.

This one-dimensional argument would work as well for locally finite graphs or long rays --- more generally, for locally compact spaces in which compact sets can be covered by finitely many compact intervals, so that sets like the above $W$ form a subbasis of the topology of uniform convergence on compact subsets, cf.~\cite[X, \S~3, n°~4, Rem.~2]{Bourbaki_TG_V}.}, which is compatible with its group structure by \cite[Th.~4]{Arens_1946}. 
\end{remark}

\section{\texorpdfstring{$\Homeo(\omega_1)$}{Homeo(omega 1)}}

\epigraph{
	Les chiffres ne s'arrêtent jamais. Peut-être parce qu'ils ne mènent nulle part.\footnotemark}{A. \textsc{Salacrou}, \textsl{L'Archipel Lenoir} ou \textsl{Il ne faut pas toucher aux choses immobiles}, 2\textsuperscript{nde} partie.}
	\footnotetext{\textit{Numbers never end. Maybe because they lead nowhere.}}

We henceforth focus on the space $X = \omega_1$, namely the first uncountable ordinal endowed with the order topology, and its homeomorphism group $G = \Homeo(\omega_1)$, endowed with the topology of pointwise convergence. (Some of the results are valid for other ordinals, see~Remark~\ref{rem:gen}.) Recall that $X$ is the linearly ordered set of all countable ordinals; a basis for its topology is given by intervals $]\alpha, \beta[$ of all ordinals $\gamma$ such that $\alpha < \gamma < \beta$ (along with the set $\{0\}$). The space $X$ is a paragon of topological pathology. For instance, it is first-countable, locally compact, normal and sequentially compact, yet it is not separable, not perfectly normal, not compact, not metrisable. Many of its fascinating properties are conveniently gathered at Example~42 of \cite{SS_1978} (where it is written as $[0, \Omega)$); the interested reader will also find more general properties of linearly ordered spaces as scattered exercises in many standard textbooks on topology (see in particular Problem 1.7.4. and the list therein in \cite{Engelking_GT}).

However, for the purpose of this mostly self-contained note, we only need the following easy facts about $X$ (in which $\alpha$ denotes a countable ordinal):
\begin{itemize}
	\item Since $X$ is well-ordered, its topology is Hausdorff and scattered. By Corollary~\ref{cor:homeoscatt}, the group $\Homeo(\omega_1)$ is thus a topological group for the topology of pointwise convergence.
	\item Intervals of the form $[0, \alpha] = [0, \alpha + 1[$ are both closed and open. In particular, any homeomorphism of $[0, \alpha]$ can be extended as an homeomorphism of $X$ fixing pointwise the interval $[\alpha + 1, \omega_1[$. (See Proposition~\ref{prop:strat} below for a kind of converse statement.)
	\item $X$ is closed by taking the supremum of countably many points, since such a supremum is still a countable ordinal (the ordinal $\omega_1$ is said to have \emph{uncountable cofinality}).
	\item By induction on $\alpha$, the ordinal $\omega^\alpha$ belongs to the derived subset $X^{(\alpha)}$ but not to $X^{(\alpha + 1)}$. (We write $\omega$ for the first infinite ordinal, namely the set of all finite ordinals.) Conversely, any point $x \neq 0$ of $X^{(\alpha)} \setminus X^{(\alpha + 1)}$ can be written as $x = x' + \omega^\alpha$ for some $x' < x$ (this follows from Cantor's normal form for ordinals, see e.g.~\cite[2.26]{Jech_2003}).
\end{itemize}

Section~\ref{sec:fphomeo} contains some results about the dynamics of individual homeomorphisms of $\omega_1$. We then prove some results towards the \enquote{wild} nature of $\Homeo(\omega_1)$, regarding its quotients in Section~\ref{sec:quotients} and its lack of Baire property in Section~\ref{sec:baire}. Its \enquote{tame} nature is investigated in Sections~\ref{sec:roelcke} (Roelcke-precompactness) and ~\ref{sec:amen} (amenability).

\subsection{Fixed points and invariant subspaces for homeomorphisms}\label{sec:fphomeo}

Let us recall that a subset $A$ of $\omega_1$ is \emph{unbounded} (or \emph{cofinal}) if for any $\alpha < \omega_1$, there is $\beta \in A$ such that $\beta \geq \alpha$. A well-known fundamental property of homeomorphisms of $\omega_1$ is their large fixed-point set, more precisely:

\begin{proposition}\label{prop:fixclosunb}
	Let $g\colon \omega_1 \rightarrow \omega_1$ be a continuous map with countable fibers. The set $\Fixp(g)$ of fixed points of $g$ is closed and unbounded.
\end{proposition}

\begin{proof}
	The set $\Fixp(g)$ is obviously closed since $g$ is continuous and $\omega_1$ is Hausdorff. Let $\alpha < \omega_1$. Define $\beta_0 = \alpha$ and, by induction, 
	\begin{equation}\label{eq:indcount}
		\beta_{n + 1} = \left(\sup \{\beta_n, g(\beta_n) \} \cup g^{-1} ([0, \beta_n]) \right) + 1
	\end{equation}
	(which is possible since $g^{-1} [0, \beta_n]$ is countable). Let $\beta$ be the limit of the increasing sequence $(\beta_n)$. By construction, $\beta_n \leq g(\beta_{n + 1}) \leq \beta_{n + 2}$, hence $\beta$ is also the limit of the sequence $(g(\beta_n))$. By continuity, this implies that $\beta$ is fixed by $g$. Since $\beta > \alpha$ and $\alpha$ was arbitrary, $\Fixp(g)$ is unbounded.
\end{proof}

Since any countable intersection of closed unbounded subsets of $\omega_1$ is still closed and unbounded (hence nonempty) \cite[8.3]{Jech_2003}, we get in particular that any countable family of homeomorphisms have uncountably many common fixed points. This observation can be generalised as follows. Let us say that a subset $A \subseteq \Homeo(\omega_1)$ is \emph{weakly precompact} if $A x$ is countable (that is, bounded) for any $x \in \omega_1$ (equivalently, $A D$ is countable for any countable subset $D$ of $\omega_1$). Obviously, countable sets are weakly precompact, and this property is stable by countable unions, subsets and translations. Moreover, if $A$ and $B$ are weakly precompact, then so are $A^{-1}$, $AB$, and the closure of $A$. In particular, the closed subgroup generated by a weakly precompact set is still weakly precompact. In addition, compact subsets are weakly precompact: indeed, if $K$ is not weakly precompact, there exists some $x \in \omega_1$ such that $K x$ is unbounded. That is, for any $\alpha \in \omega_1$, we can find $g_\alpha \in K$ such that $g_\alpha (x) \geq \alpha$. It is therefore impossible to extract a converging subnet of the net $(g_\alpha)$, hence $K$ cannot be compact.

The point of this notion is that weakly precompact subsets of $\Homeo(\omega_1)$ are \enquote{small} in the following sense.

\begin{proposition}
	Let $A$ be a weakly precompact subset of $\Homeo(\omega_1)$. Then the set of common fixed points $\Fixp(A) = \bigcap_{g \in A} \Fixp (g)$ is closed and unbounded.
\end{proposition}

\begin{proof}
	The proof is completely similar to that of Proposition~\ref{prop:fixclosunb}, replacing the inductive definition \eqref{eq:indcount} thereof by
	\begin{equation*}
		\beta_{n + 1} = \left(\sup A \beta_n \cup A^{-1} [0, \beta_n]\right) + 1,
	\end{equation*}
	the set of which we take the supremum being countable thanks to the weak precompactness of $A$ (and $A^{-1}$).
\end{proof}

In particular, Polish or $\sigma$-compact subgroups of $\Homeo(\omega_1)$ have uncountably many common fixed points.

\begin{remark}
	Since a basis of identity neighbourhoods of $G$ is given by fixators of finite sets, we can equivalently say that a subset $A$ is weakly precompact if for any identity neighbourhood $U$, there is a countable subset $D \subseteq G$ such that $A \subseteq D U$. The same definition with \enquote{countable} replaced by \enquote{finite} gives precompact subsets, hence the name.
\end{remark}

We now prove another useful property of homeomorphisms of $\omega_1$, namely that $\omega_1$ can be \enquote{stratified} into invariant subsets.

\begin{proposition}\label{prop:strat}
	Let $g\colon \omega_1 \rightarrow \omega_1$ be a continuous map (respectively, a homeomorphism). The set $S'_g$ (resp., $S_g$) of all $\alpha \in \omega_1$ such that $g([0, \alpha]) \subseteq [0, \alpha]$ (resp., $g([0, \alpha]) = [0, \alpha]$) is closed and unbounded.
\end{proposition}

\begin{proof}
	We only need to prove the statement about continuous maps since $S_g = S'_g \cap S'_{g^{-1}}$ for bijective maps. Let $\beta \in \omega_1$ and define by induction $\alpha_n$ via $\alpha_0 = \beta$ and $\alpha_{n + 1} = \sup g([0, \alpha_n])$. Define moreover $\alpha = \sup \alpha_n$ and let us show that $\alpha \in S'_g$. Let $\gamma < \alpha$. By definition, there is some $n$ such that $\gamma \leq \alpha_n$. Hence $g(\gamma) \in g([0, \alpha_n])$ and $g(\gamma) \leq \alpha_{n + 1} \leq \alpha$. This shows that $g([0, \alpha[) \subseteq [0, \alpha]$. If $\alpha$ is a limit ordinal, then $g(\alpha) \in [0, \alpha]$ by continuity of $g$. If it is not, then there must exist some $n$ such that $\alpha = \alpha_n$, hence $g(\alpha) \leq \alpha_{n + 1} \leq \alpha$. In both cases, $g([0, \alpha]) \subseteq [0, \alpha]$.
\end{proof}

\subsection{Quotients}\label{sec:quotients}

Since derived subspaces are preserved by any homeomorphism, $\Homeo(\omega_1)$ has uncountably many normal closed subgroups, namely the fixators of $\omega_1^{(\alpha)}$ for any countable ordinal $\alpha$. Yet none of its quotients can be \enquote{small}, for instance metrisable. The following simple observation will make the proofs easier.

\begin{lemma}\label{lem:invdense}
	Let $A$ be a subset of $\Homeo(\omega_1)$ which is invariant by conjugation and contains the fixator $\Fixa\{x_i\}$ of some countable family of points $\{x_i\}$. Then $A$ is dense in $\Homeo(\omega_1)$.
\end{lemma}

\begin{proof}
	Let $g \in \Homeo(\omega_1)$ and $y_1, \dots, y_k \in \omega_1$; we need to show that there exists $h \in A$ such that $g(y_i) = h(y_i)$ for $i=1, \dots, k$. By Proposition~\ref{prop:strat}, there is $\alpha > \max \{y_1, \dots, y_k\}$ such that $g([0, \alpha]) = [0, \alpha]$. Let $h$ be the homeomorphism agreeing with $g$ on $[0, \alpha]$ (hence in particular $g(y_i) = h(y_i)$) and fixing the complement $]\alpha, \omega_1[$. Let $k$ be any homeomorphism such that $\alpha < \min \{k(x_i)\}$ (such a homeomorphism exists: we can for instance swap $[0, \beta]$ and $[\beta + 1, \beta \cdot 2]$ for any limit ordinal $\beta$ such that $\beta > \sup \left(\{\alpha\} \cup \{x_i\}\right)$). Then $h \in \Fixa \{k(x_i)\} = k (\Fixa \{x_i\}) k^{-1} \subseteq A$.
\end{proof}

\begin{theorem}\label{th:propum}
	Any continuous morphism $\phi$ from $\Homeo(\omega_1)$ to a metrisable group $H$ is trivial.
\end{theorem}

\begin{proof}
	Let $\{U_n\}$ be a countable identity neighbourhood basis of $H$ and $V_n = \phi^{-1} (U_n)$. Being open, each set $V_n$ must contain some $\Fixa F_n$, where $F_n$ is a finite family of points in $\omega_1$. Therefore, $\bigcap_n V_n = \ker \phi$ contains the fixator of the countable family $\bigcup_n F_n$. Since $\phi$ is continuous, this kernel is closed, hence $\ker \phi = \Homeo(\omega_1)$ by Lemma~\ref{lem:invdense}.
\end{proof}

In particular, any continuous isometric action of $\Homeo(\omega_1)$ on a \emph{separable} metric space $Z$ is trivial, since such an action gives rise to a continuous morphism to $\Isom (Z)$ (endowed with the topology of pointwise convergence), and the latter is metrisable if $Z$ is separable. The separability assumption is unavoidable : by Proposition~\ref{prop:topscattpw}, $\Homeo(\omega_1)$ admits a natural continuous unitary representation into $\ell^2 (\omega_1)$ (given by $g \cdot \delta_x = \delta_{gx}$). However, we will show below (Theorem~\ref{th:roelcke} and the discussion thereafter) that any continuous isometric action of $\Homeo(\omega_1)$ has \emph{bounded} orbits.

\begin{proposition}
	Any proper closed subgroup of $\Homeo(\omega_1)$ has uncountable index.
\end{proposition}

\begin{proof}
	Let $H$ be a closed subgroup of $G = \Homeo(\omega_1)$ with countable index and let $\{g_n\}$ be a representative system for the right cosets of $H$, with $g_0 \in H$. Since $H$ is closed and $G$, as a Hausdorff topological group, is completely regular, there exists, for any $n \neq 0$, an open subgroup $U_n$ such that $g_n \not\in U_n H$. Once again, the intersection of all the $U_n$'s contains the fixator $F = \Fixa \{x_i\}$ of some countable family of points in $\omega_1$. Let $f \in F$ and write $f = g_i h$ for some index $i$ and $h \in H$. Then $g_i = f h^{-1} \in F H \subseteq U_n H$ for all $n \neq 0$, hence $i = 0$. Therefore, $F \subseteq H$.

	Let now $N$ be the core of $H$. Since $N = \bigcap_n g_n H g_n^{-1}$, it contains $\bigcap_n g_n F g_n^{-1}$, which is the fixator of the countable family $\{g_n x_i\}$. Hence $N = H = \Homeo(\omega_1)$ by Lemma~\ref{lem:invdense}.
\end{proof}

\begin{remark}
    We prove in the companion paper that the only closed normal subgroups of $\Homeo(\omega_1)$ are the fixators of the derived subsets, $\Fixa(\omega_1^{(\alpha)})$ \cite[Remark~11 and Proposition~18]{Gheysens_dynscatt}.
\end{remark}

\subsection{Baire property}\label{sec:baire}

The group $\Homeo(\omega_1)$ is not Baire. Consider for instance, for any positive integer $n$, the set $T_n = \bigcup_{n \leq k < \omega} \Fixa (k)$. This open set is dense. Indeed, for any $g \in \Homeo(\omega_1)$ and $x_1, \dots, x_m \in \omega_1$, we can find an integer $k \geq n$ such that $k \neq x_i$ and $k \neq g(x_i)$ for all $i= 1, \dots, m$. Therefore, the homeomorphism $h = g \circ (k,\ g^{-1} (k))$ belongs to $T_n$ and satisfies $h(x_i) = g(x_i)$ for all $i$.

However, if $g \in \bigcap_{n < \omega} T_n$, then there exists infinitely many integers $k_j$ such that $g(k_j) = k_j$, hence $g(\omega) = \omega$ by continuity. Consequently, the countable intersection of the $T_n$'s is contained in the closed subgroup $\Fixa(\omega)$, thus is not dense: $\Homeo(\omega_1)$ is not Baire.

This stands in contrast to the symmetric group $\Sym(\aleph_1)$, which is Baire (see Appendix~\ref{app:symbaire}).

\subsection{Roelcke-precompactness}\label{sec:roelcke}

Homeomorphisms of a space have to preserve derived subsets. For $\Homeo(\omega_1)$, this is the only obstruction to $k$-transitivity, namely:

\begin{lemma}\label{lem:oligom}
Let $k$ be any positive integer. Let $\alpha_1, \dots, \alpha_k$ be $k$ countable ordinals (not necessarily pairwise distinct). For $i = 1, \dots, k$, let $x_i$ and $y_i$ be two points in $Y^{(\alpha_i)} = \omega_1^{(\alpha_i)} \setminus \omega_1^{(\alpha_i + 1)}$. Assume that $x_i \neq x_j$ and $y_i \neq y_j$ whenever $i \neq j$. Then there exists $g \in \Homeo(\omega_1)$ such that $g(x_i) = y_i$ for each $i$.
\end{lemma}

\begin{proof}
 By induction on $k$, it is enough to show that, for any positive integer $n$ and any points $x_1, \dots, x_n \in \omega_1$, the fixator $F = \Fixa\{x_1, \dots, x_n\}$ acts transitively on $Y^{(\alpha)} \setminus \{x_1, \dots, x_n\}$ for any countable ordinal $\alpha$. Let then $y$ and $z$ be two distinct points of the latter set and let us show that there is $g \in F$ such that $g(y) = z$. We may assume that $y < z$ (up to switching $g$ and its inverse). Observe also that the existence of $g$ is trivial if $\alpha = 0$: in that case, $y$ and $z$ are isolated in $\omega_1$, hence the transposition $(y, z)$ is a homeomorphism.

	We can therefore assume that $\alpha > 0$, hence $y$ and $z$ are of the form $y = y' + \omega^\alpha$ and $z = z' + \omega^\alpha$, for some $y'$ and $z'$ such that $y' < y \leq z' < z$. Moreover, since, for any $\gamma < \omega^\alpha$, we have $\gamma + \omega^\alpha = \omega^\alpha$, there are infinitely many such $y'$ and $z'$, hence we can find some such that the intervals $]y', y]$ and $]z', z]$ do not meet $\{x_1, \dots, x_n\}$.

	Since $y' + \omega^\alpha = y$ and $z' + \omega^\alpha = z$, there is a (unique) isomorphism of ordered sets $h\colon ]y', y] \rightarrow ]z', z]$ (both intervals are isomorphic to $]0, \omega^\alpha]$). Since these intervals are closed, open, and disjoint, the map $g$ that switches $x$ and $h(x)$ for any $x \in ]y', y]$ and fixes any point outside $]y', y] \cup ]z', z]$ is a homeomorphism of $\omega_1$. By construction, all the points $x_i$'s are outside these two intervals, hence $g \in F$ as required.
\end{proof}

For the next theorem, we recall that a topological group $H$ is said \emph{Roelcke-precompact} if it is precompact for its lower uniform structure, that is: for any identity neighbourhood $U$, there is a finite subset $F$ of $H$ such that $H = UFU$.

\begin{theorem}\label{th:roelcke}
	The group $\Homeo(\omega_1)$ is Roelcke-precompact.
\end{theorem}

\begin{proof}
	Let $U = \Fixa\{x_1, \dots, x_n\}$ be any basic identity neighbourhood of $G = \Homeo(\omega_1)$. For any partially defined injective map $\sigma\colon \{1, \dots, n\} \rightarrow \{1, \dots, n\}$, let us choose, thanks to Lemma~\ref{lem:oligom}, a homeomorphism $h_\sigma$ such that $h_\sigma (x_i) = x_{\sigma(i)}$ whenever $i \in \dom \sigma$ is such that $x_i$ and $x_{\sigma (i)}$ are in the same set $Y^{(\alpha)}$. We then claim that $G = UFU$, where $F$ is the finite set of all $h_\sigma$. Indeed, for any $g \in G$, let $\tau$ be the partially defined injective map of $\{1, \dots, n\}$ which is given by $g(x_i) = x_{\tau(i)}$. In particular, for any $i \not\in \dom \tau$, both $h_\tau (x_i)$ and $g(x_i)$ avoid the set $\{x_1, \dots, x_n\}$, hence we can find, by Lemma~\ref{lem:oligom}, a homeomorphism $u \in U$ such that $u(h_\tau (x_i)) = g(x_i)$ for all $i \not\in \dom \tau$. Therefore, $g^{-1} u h_\tau \in U$, hence $g \in UFU$.
\end{proof}

As a general consequence of Roelcke-precompactness, $\Homeo(\omega_1)$ enjoys the following rigidity property: any continuous isometric action on a metric space has bounded orbits \cite[Prop.~2.3]{Tsankov_2012}. Note that this property is unrelated to the absence of morphisms to metrisable \emph{groups} (Theorem~\ref{th:propum}), as shown in Appendix~\ref{app:tree}.

\begin{remark}
	The above Lemma and Theorem are very reminiscent of the well-known link between oligomorphy and Roelcke-precompactness (see for instance \cite[Th.~2.4]{Tsankov_2012}, where countability assumptions are only used to stay in the Polish realm). Indeed, Lemma~\ref{lem:oligom} shows that for any $n$-tuple of countable ordinals $\alpha_1, \dots, \alpha_n$, the action of $\Homeo(\omega_1)$ on $Y^{(\alpha_1)} \cup \dots \cup Y^{(\alpha_n)}$ is oligomorphic. Hence $\Homeo(\omega_1)$ is an (uncountable) inverse limit of oligomorphic groups and therefore Roelcke-precompact.
\end{remark}

\subsection{Amenability}\label{sec:amen}

We recall that, for a general topological group $H$, amenability is defined as the existence of a left-invariant mean on the space $\Contrucb(H)$ of all bounded right-uniformly continuous maps from $H$ to $\Reels$. Equivalently, $H$ is amenable if any jointly continuous affine action of $H$ on a nonempty compact convex subset of a locally convex space has a fixed point \cite[4.2]{Rickert_1967}.

We will need the following well-known fact, for which we provide a proof for the convenience of the reader.

\begin{lemma}
Let $(G_i)_{i \in I}$ be a family of amenable topological groups. The direct product $\prod_{i \in I} G_i$ is amenable (as a topological group).
\end{lemma}

\begin{proof}
Since amenability is stable by extensions, the direct product of \emph{finitely many} amenable groups is amenable by induction. Therefore, the \emph{restricted} product $\bigoplus_i G_i$, which is a directed union of amenable groups, is itself amenable. Lastly, the restricted product is dense in the direct product, hence the latter is amenable. (Stability of amenability by extensions and by directed unions, as well as inheritance from a dense subgroup, are easily shown via the fixed-point characterisation of amenability; they can be found for instance in 4.8, 4.7 and 4.5 of \cite{Rickert_1967}.)
\end{proof}

(Note that it is essential to consider the product topology on $\prod_i G_i$. Indeed, the direct product of infinitely many amenable groups is usually not amenable as a discrete group; for instance, $\prod_n \operatorname{SL}_2 (\Int / n \Int)$ contains a free subgroup, hence is not amenable as an abstract group, see e.g.~\cite[1.2.3 and 1.2.5]{Greenleaf_1969}.)

\begin{theorem}\label{th:amen}
	The group $\Homeo(\omega_1)$ is amenable.
\end{theorem}

\begin{proof}
	By Proposition~\ref{prop:topscattpw}, the topological group $G = \Homeo(\omega_1)$ is isomorphic to a subgroup of $\Sym(\aleph_1)$. Let $\overline{G}$ be the closure of $G$ in $\Sym(\aleph_1)$. Since, for any \emph{topological} groups $H < H'$, the uniform structure induced on $H$ by the right uniform structure of $H'$ is the same as the right uniform structure associated to the topology induced on $H$ by $H'$ (see \cite[3.24]{RD_1981}), there is a $G$-equivariant isometry between $\Contrucb (G)$ and $\Contrucb (\overline{G})$. Therefore, to prove the amenability of $G$, we only need to prove that of $\overline{G}$. But, thanks to Lemma~\ref{lem:oligom}, $\overline{G}$ is nothing but the direct product of groups of the form $\Sym\left(Y^{(\beta)}\right)$, where $Y^{(\beta)} = \omega_1^{(\beta)} \setminus \omega_1^{(\beta + 1)}$. All of these are isomorphic to $\Sym(\aleph_1)$, which contains a dense locally finite subgroup, hence is amenable. Therefore, $\overline{G}$ is amenable by the above Lemma.
\end{proof}

\begin{remark}
The group $G$ is \emph{not} closed in $\Sym(\aleph_1)$ (that is, $G \neq \overline{G}$). For instance, the permutation that swaps $n$ et $\omega + n$ for any $1 \leq n < \omega$ is in $\overline{G}$ but is not continuous at the point $\omega$, hence is not in $G$. (Moreover, it can be checked that $\overline{G}$ is actually Baire, see~Appendix~\ref{app:symbaire}.)
\end{remark}

\begin{remark}
	Similarly to the symmetric group $\Sym(\aleph_1)$, the group $\Homeo(\omega_1)$ admits an action on the set of all linear orders on $\aleph_1$ (compare with \cite[\S~2]{GW_2002}), which has no fixed point owing to Lemma~\ref{lem:oligom}. This set can be seen as a closed subset of $2^{\aleph_1 \times \aleph_1}$ and is therefore compact. By Proposition~\ref{prop:topscattpw}, this action is jointly continuous. Hence $\Homeo(\omega_1)$ is not \emph{extremely} amenable. (Extreme amenability is defined by removing the \enquote{convex} assumption in the fixed-point property, namely: $H$ is \emph{extremely amenable} if any jointly continuous action of $H$ on a nonempty compact space has a fixed point.)

Actually, more can be said about continuous actions on compact sets: the \emph{universal minimal flow} of $\Homeo(\omega_1)$ is computed in the companion paper \cite{Gheysens_dynscatt} (see in particular Remark~11 and Theorem~15 therein) --- and was also obtained independently by Basso and Zucker in \cite[\S~9]{BZ_2020}.
\end{remark}

\section{Further remarks}

\begin{remark}
    When it comes to exhibiting pathological phenoma in topology, the compact space $\omega_1 + 1$ (more commonly written as $[0, \omega_1]$ or $[0, \Omega]$), which is also the one-point compactification of $\omega_1$, is another first-rate space (Example~43 in \cite{SS_1978}). For this note, working instead with its sibling $\omega_1$ is purely a cosmetic choice, since both spaces have the \emph{same} homeomorphism group. Indeed, in $[0, \omega_1]$, the point $\omega_1$ is fixed by any homeomorphism, because it is the only point belonging to the $\omega_1$-st derived subset of $[0, \omega_1]$. Therefore, the canonical topological embedding $\Homeo(\omega_1) \rightarrow \Homeo(\omega_1 + 1)$, given by extending any homeomorphism of $\omega_1$ to its one-point compactifiation, is here a topological isomorphism.
\end{remark}

\begin{remark}\label{rem:gen}
	We focused on the ordinal $\omega_1$ for the sake of exposition, but a careful analysis of the proofs shows that the results are valid for other ordinals:
	\begin{itemize}
		\item Sections~\ref{sec:fphomeo} and~\ref{sec:quotients} hold for any ordinal with uncountable cofinality;
		\item Baire property would fail for any ordinal $> \omega \cdot 2$;
		\item Roelcke-precompactness and amenability hold for all ordinals.
	\end{itemize}
For the latter point, two paths are available. The ad hoc proof we gave here for $\omega_1$ (starting with the key Lemma~\ref{lem:oligom}) can be followed again for other ordinals. The only annoyance is to check if the words \enquote{countable ordinal} were meant for \enquote{point in $\omega_1$} or for \enquote{ordinal below the Cantor--Bendixson rank of $\omega_1$}. Alternatively, and perhaps more satisfactorily, we provide in the companion paper an argument (summarized in \cite[Corollary~23]{Gheysens_dynscatt}) that works for any locally compact scattered space, sparing us the resort to the specific structure of ordinals.
\end{remark}

\begin{remark}
	A connected cousin of the space $\omega_1$ is the \emph{closed long ray}, namely the space $\clr = \omega_1 \times [0, 1[$ endowed with the lexicographic order topology, see for instance \cite[\S~1.2]{Gauld_2014}. The group $\Homeo(\clr)$ is also a topological group when endowed with the topology of pointwise convergence, which agrees with the topology of uniform convergence on compact sets (compare with Remark~\ref{rem:pwunifcomp}). In contrast with $\Homeo(\omega_1)$, the group $\Homeo(\clr)$ is path-connected \cite[Cor.~2]{Gauld_1991}.
\end{remark}

\begin{remark}\label{rem:actionnotcontinuous}
	For any positive integer $n$, let $g_n \in \Homeo(\omega_1)$ be the transposition $(n, \omega + n)$. Then the sequence $g_n$ converges to the identity, but $g_n (n)$ converges to $\omega \cdot 2$. Hence despite Proposition~\ref{prop:topscattpw}, the action map $\Homeo(X) \times X \rightarrow X$ is in general \emph{not} continuous when $X$ is scattered but not discrete. (Helmer provided other examples where the topology of pointwise convergence is compatible with the structure of some group of homeomorphisms but does not yield the continuity of the action map, see~\cite[Ex.~13 and 14]{Helmer_1980}.)
\end{remark}

\appendix
\section{\texorpdfstring{$\Sym(X)$}{Sym(X)} is a Baire group}\label{app:symbaire}

In contrast with $\Homeo(\omega_1)$, we show here that $\Sym(X)$, endowed with the topology of pointwise convergence, is a Baire group, whatever the cardinality of $X$. When $X$ is countable, this is encompassed in the well-known fact that $\Sym(X)$ is Polish. However, when $X$ is uncountable, the classical Baire category theorem does not apply since $\Sym(X)$ is then neither metrisable or locally compact. The group $\Sym(X)$ is nonetheless still \emph{sievable} in Choquet's sense \cite{Choquet_1958_tamis}, a notion we now recall.

A \emph{sieve} (in French: \emph{tamis}) on a topological space is a binary relation $\sqsubseteq$ on \emph{nonempty} open subsets such that:
\begin{enumerate}
	\item $O \sqsubseteq O'$ implies $O \subseteq O'$;
	\item $O_1 \subseteq O_2 \sqsubseteq O_3 \subseteq O_4$ implies $O_1 \sqsubseteq O_4$;
	\item for any $O$, there is $O'$ such that $O' \sqsubseteq O$;
	\item if $(O_n)$ is a sequence such that $O_{n+1} \sqsubseteq O_n$ for all $n$, then the intersection $\bigcap_n O_n$ is nonempty 
\end{enumerate}
($O$, $O'$ and $O_n$ above are all assumed to be nonempty open subsets). It is easy to show that a sievable space (that is, a space admitting a sieve) is Baire. Moreover, locally compact spaces and completely metrisable spaces are sievable \cite{Choquet_1958_tamis}. The remarkable feature of sieves is that an arbitrary product of sievable spaces is still sievable---whereas the product of two Baire spaces is not necessarily Baire \cite{Cohen_1976}, even if they are topological groups \cite{Valdivia_1985}.

\begin{theorem}
	The topological group $\Sym(X)$ is sievable.
\end{theorem}

\begin{proof}
We may assume that $X$ is uncountable. Let $\sqsubseteq$ be the binary relation on nonempty open subsets of $\Sym(X)$ defined by: $O_1 \sqsubseteq O_2$ if there are finitely many $x_1, \dots, x_k \in X$ and finitely many finite subsets $Y_1, \dots, Y_k \subset X$ such that
	\begin{equation}\label{eq:sievesym}
		O_1 \subseteq \{g \in \Sym(X)\ |\ g(x_i) \in Y_i\ \text{for all}\ i=1, \dots, k\} \subseteq O_2.
	\end{equation}
	The first two properties of the definition of a sieve are obviously satisfied by $\sqsubseteq$, and the third one is ensured by the definition of the topology of pointwise convergence on $\Sym(X)$. To check the fourth one, let $(O_n)$ be a $\sqsubseteq$-decreasing sequence. Let $x_1^n, \dots, x_{k_n}^n \in X$ and $Y_1^n, \dots, Y_{k_n}^n \subset X$ be two finite families witnessing $O_{n + 1} \sqsubseteq O_n$ as in \eqref{eq:sievesym}. Thanks to the first inclusion in \eqref{eq:sievesym}, we may always assume, up to relabelling the indices, that, for $n < m$, we have both
	\begin{itemize}
		\item $k_n \leq k_m$;
		\item for $i \leq k_n$, $x_i^n = x_i^m$ and $Y_i^m \subset Y_i^n$.
	\end{itemize}
	Let us then simply write $x_i$ for $x_i^n$ as soon as $i \leq k_n$ and define $Y_i = \bigcap_{n:\ i \leq k_n} Y_i^n$. Each $Y_i$ is nonempty since it is the decreasing intersection of nonempty finite sets. Let then $F$ be the set of all $g \in \Sym(X)$ such that $g(x_i) \in Y_i$ for all $i$. Obviously, $F$ is contained in each $O_n$, let us then prove that it is nonempty.

	For any finite subset of indices $J$, there exists $n$ such that $J \subseteq \{1, \dots, k_n\}$ and $Y_i = Y_i^n$ for all $i \in J$. Therefore, there exists $g \in \Sym(X)$ such that $g(x_i) \in Y_i$ for all $i \in J$, namely any $g$ in $O_{n + 1}$. This proves that the selection problem defining $F$ satisfies Hall's condition, hence we can find an injective map $h$ from $\{x_i\ |\ i \in \Nat\}$ to $X$ such that $h(x_i) \in Y_i$ for all $i$ \cite{Hall_1948}. Since $X$ is uncountable, such an injective map can be extended to a bijection of $X$, which then belongs to $F$. This concludes the proof that $\sqsubseteq$ defines a sieve on $\Sym(X)$.
\end{proof}

\section{Regular trees of uncountable degree}\label{app:tree}

	Let us consider the following two properties of a topological group $G$, both enjoyed by $\Homeo(\omega_1)$:
	\begin{itemize}
		\item[\textbf{\propum}] Any continuous morphism from $G$ to a metrisable group is trivial.
		\item[\textbf{(OB)}] Any continuous isometric action of $G$ on a metric space has bounded orbits.
	\end{itemize}
	As we observed after Theorem~\ref{th:propum}, property \propum\ implies that any continuous isometric action of $G$ on a \emph{separable} metric space is trivial. It is then tempting to conjecture a link between properties \propum\ and (OB)---yet there is none. On the one hand, many metrisable (or indeed discrete) groups enjoy property (OB), see e.g.~\cite{Bergman_2006, Cornulier_2006, Rosendal_2009}.

	On the other hand, let $T$ be a regular tree of uncountable degree. (For instance, we can take as vertex set the set of all finite words in an uncountable alphabet $A$, and two vertices are joined by an edge if one is obtained from the other by adding a single letter at the end of the word.) Endowed with the path-metric, the vertex set $V(T)$ of $T$ is a nonseparable metric space. The topology of pointwise convergence on $V(T)$ is compatible\footnote{For any of the following reasons: $V(T)$ is topologically discrete (balls of radius $<1$ are singletons); a tree is a first-order structure; or $\Aut(T)$ is equicontinuous since it acts by isometries.} with the group structure of the automorphism group $\Aut(T)$ and makes the natural action on $V(T)$ continuous.

	Let $G = \Aut^+ (T)$ be the subgroup of $\Aut(T)$ generated by the fixators of edges. It does not enjoy property~(OB) since it has unbounded orbits in $T$ (indeed, it has only two orbits). Let us show that $G$ enjoys property \propum. By \cite[4.5]{Tits_1970}, $G$ is abstractly simple, hence a fortiori topologically simple. Therefore, we only need to show that there is no \emph{injective} continuous morphism from $G$ to a metrisable group. But such a morphism would allow to write its kernel, namely a singleton, as a countable intersection of open sets. Since a basis of identity neighbourhoods of $G$ is given by fixators of finitely many points, such an intersection contains the fixator of countably many points, which is a nontrivial subgroup since any vertex of $T$ has uncountably many neighbours. Hence a continuous morphism from $G$ to a metrisable group cannot be injective and thus $G$ enjoys property \propum.

\bibliographystyle{../../BIB/amsalpha}
\bibliography{../../biblio/biblio}

\end{document}